\documentclass[11pt]{amsart}
\usepackage{amssymb,amsmath,mathrsfs}
\usepackage{enumerate,tikz}
\usepackage{hyperref,soul}

\usetikzlibrary{patterns}

\newtheorem{thm}{Theorem}[section]
\newtheorem{lem}[thm]{Lemma}
\newtheorem*{conjecture}{Conjecture}
\numberwithin{equation}{section}

\def\F{\mathscr{F}}
\def\A{\mathscr{A}}
\def\Av{\mathsf{Av}}
\def\IF{I\!F}
\def\oeis#1{\cite[#1]{Sloane}}
\def\fpattern{%
   \tikz[scale=0.2,baseline=5pt]{\draw[gray!80] (0,0) grid (3,3); \draw[fill] (0.5,1.5) circle (0.18); \draw[fill] (1.5,2.5) circle (0.18); \draw[fill] (2.5,0.5) circle (0.18); \draw[thick, black] (1,0)--(1,3); \draw[thick, black] (0,1)--(3,1)}}

\def\R{\rule[-1ex]{0ex}{3.6ex}}

\begin{document}
\title{On pattern-avoiding Fishburn permutations}
\author[Gil]{Juan B. Gil}
\address{Penn State Altoona\\ 3000 Ivyside Park\\ Altoona, PA 16601}
\email{jgil@psu.edu}

\author[Weiner]{Michael D. Weiner}
\email{mdw8@psu.edu}

%\keywords{}

\begin{abstract}
The class of permutations that avoid the bivincular pattern $(231, \{1\},\{1\})$ is known to be enumerated by the Fishburn numbers. In this paper, we call them {\em Fishburn permutations} and study their pattern avoidance. For classical patterns of size 3, we give a complete enumerative picture for regular and indecomposable Fishburn permutations. For patterns of size 4, we focus on a Wilf equivalence class of Fishburn permutations that are enumerated by the Catalan numbers. In addition, we also discuss a class enumerated by the binomial transform of the Catalan numbers and give conjectures for other equivalence classes of pattern-avoiding Fishburn permutations.
\end{abstract}

\maketitle

%%%%%%%%%%%%%%%%%%%%%%%%%%%%%%%%%%%%%%%%%%%%
\section{Introduction}

Motivated by a recent paper by G.~Andrews and J.~Sellers \cite{AS16}, we became interested in the Fishburn numbers $\xi(n)$, defined by the formal power series
\[  \sum_{n=0}^\infty \xi(n) q^n = 1+\sum_{n=1}^{\infty}\prod_{j=1}^n (1-(1-q)^j). \]
They are listed as sequence A022493 in \cite{Sloane} and have several combinatorial interpretations. For example, $\xi(n)$ gives the:
\begin{itemize}
\item[$\triangleright$] number of linearized chord diagrams of degree $n$,
\item[$\triangleright$] number of unlabeled $(2+2)$-free posets on $n$ elements,
\item[$\triangleright$] number of ascent sequences of length $n$,
\item[$\triangleright$] number of permutations in $S_n$ that avoid a certain bivincular pattern.%
\footnote{$S_n$ denotes the set of permutations on $[n]=\{1,\dots,n\}$.}
\end{itemize}
For more on these interpretations, we refer the reader to \cite{BM+10} and the references there in.

In this note, we are primarily concerned with the aforementioned class of permutations. That they are enumerated by the Fishburn numbers was proved in \cite{BM+10} by Bousquet-M{\'e}lou, Claesson, Dukes, and Kitaev, where the authors introduced bivincular patterns (permutations with restrictions on the adjacency of positions and values) and gave a bijection to ascent sequences. More specifically, a permutation $\pi\in S_n$ is said to contain the bivincular pattern $(231,\{1\},\{1\})$ if there are positions $i<k$ with $\pi(i)>1$, $\pi(k) = \pi(i)-1$, such that the subsequence $\pi(i)\pi(i+1)\pi(k)$ forms a 231 pattern. Such a bivincular pattern may be visualized by the plot

\medskip
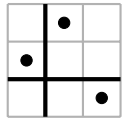
\begin{center}
\tikz[scale=0.5]{
\draw[thick,gray!60] (0,0) grid (3,3); 
\draw[fill] (0.5,1.5) circle (0.15); \draw[fill] (1.5,2.5) circle (0.15); \draw[fill] (2.5,0.5) circle (0.15); 
\draw[ultra thick, black] (1,0)--(1,3); \draw[ultra thick, black] (0,1)--(3,1)}
\end{center}
where bold lines indicate adjacent entries and gray lines indicate an elastic distance between the entries. 
%For $n=4$, there are 9 permutations containing such a pattern:
%\[ 1342,\; 2314,\; 2341,\; 2413,\; 2431,\; 3241,\; 3412,\; 3421,\; 4231. \] 

We let $\F_n$ denote the class of permutations in $S_n$ that avoid the pattern $\fpattern$, and since $|\F_n|=\xi(n)$ (see \cite{BM+10}), we call the elements of $\F = \bigcup_n \F_n$ {\em Fishburn permutations}. 
Further, we let $\F_n(\sigma)$ denote the class of Fishburn permutations that avoid the pattern $\sigma$.

Our goal is to study $F_n(\sigma)=\left|\F_n(\sigma)\right|$ for classical patterns of size 3 or 4. In Section~\ref{sec:length3patterns}, we give a complete picture for regular and indecomposable Fishburn permutations that avoid a classical pattern of size 3. Table~\ref{tab:3Patterns} and Table~\ref{tab:3PatternsInd} at the end of that section provide a summary of our findings. In Section~\ref{sec:length4patterns}, we discuss patterns of size 4, focusing on a Wilf equivalence class of Fishburn permutations that are enumerated by the Catalan numbers $C_n$ (see Table~\ref{tab:CatalanClass}). We also prove the formula $F_n(1342) = \sum_{k=1}^n \binom{n-1}{k-1} C_{n-k}$, and  conjecture two other equivalence classes (see Table~\ref{tab:4Patterns}). Finally, in Section~\ref{sec:remarks}, we briefly discuss indecomposable Fishburn permutations that avoid a pattern of size 4. In Table~\ref{tab:indLength4}, we make some conjectures based on our limited preliminary computations.  

\noindent
{\bf Basic notation.} Permutations will be written in one-line notation. Given two permutations $\sigma$ and $\tau$ of sizes $k$ and $\ell$ respectively, their direct sum $\sigma\oplus\tau$ is the permutation of size $k + \ell$ consisting of $\sigma$ followed by a shifted copy of $\tau$. Similarly, their skew sum $\sigma\ominus\tau$ is the permutation consisting of $\tau$ preceded by a shifted copy of $\sigma$. For example, $312\oplus 21 = 31254$ and $312\ominus 21 = 53421$.

A permutation is said to be indecomposable if it cannot be written as a direct sum of two nonempty permutations.

$\Av_n(\sigma)$ denotes the class of permutations in $S_n$ that avoid the pattern $\sigma$. It is well known that if $\sigma\in S_3$ then $|\Av_n(\sigma)|=C_n$, where $C_n$ is the Catalan number $\frac{1}{n+1}\binom{2n}{n}$, see e.g.\ \cite{Kitaev11}.

%%%%%%%%%%%%%%%%%%%%%%%%%%%%%%%%%%%%%%%%%%%%
\section{Avoiding patterns of size 3}
\label{sec:length3patterns}

Clearly, $\Av_n(231) \subset \F_n$. Now, since every Fishburn permutation that avoids the classical pattern $231$ is contained in the set of $231$-avoiding permutations, we get
\begin{equation}\label{eq:p231} 
  \F_n(231) = \Av_n(231), \text{ and so } F_n(231)=C_n.
\end{equation}
Enumeration of the Fishburn permutations that avoid the other five classical patterns of size 3 is less obvious.
\begin{thm} \label{thm:123av}
For $\sigma\in\{123,132,213,312\}$, we have $F_n(\sigma) = 2^{n-1}$.
\end{thm}
\begin{proof}
First of all, note that for every $\sigma$ of size 3, we have $F_1(\sigma)=1$ and $F_2(\sigma)=2$.

\medskip\noindent
{\sc Case} $\sigma=132$: If $\pi\in \F_{n-1}(132)$, the permutations $1\ominus\pi$ and $\pi\oplus 1$ are both in $\F_n(132)$. On the other hand, if $\tau$ is a permutation in $\Av_n(132)$ with $\tau(i)=n$ for some $1<i<n$, then we must have $\tau(j)>\tau(k)$ for every $j\in\{1,\dots,i-1\}$ and $k\in\{i+1,\dots,n\}$. Thus $n-i=\tau(k')$ for some $k'>i$ and $n-i+1=\tau(j')$ for some $j'<i$. But this violates the Fishburn condition since $n-i+1$ is the smallest value left from $n$ and must therefore be part of an ascent in $\tau(1)\cdots\tau(i-1)\,n$. 
In other words, $\F_{n}(132)$ is the disjoint union of the sets $\{1\ominus\pi: \pi\in \F_{n-1}(132)\}$ and $\{\pi\oplus 1: \pi\in \F_{n-1}(132)\}$. Thus 
\[ F_n(132) = 2F_{n-1}(132) \;\text{ for } n>1,\]
which implies $F_n(132) = 2^{n-1}$.

\medskip\noindent
{\sc Case} $\sigma=123$: For $n>2$, the permutation $(n-1)(n-2)\cdots21n$ is the only permutation in $\F_{n}(123)$ that ends with $n$, and if $\pi\in \F_{n-1}(123)$, then $1\ominus\pi \in \F_{n}(123)$.

Assume $\tau\in \F_{n}(123)$ is such that $\tau(i)=n$ for some $1<i<n$. Since $\tau$ avoids the pattern 123, we must have $\tau(1)>\tau(2)>\cdots>\tau(i-1)$. Moreover, the Fishburn condition forces $\tau(i-1)=1$, which implies $\tau(i+1)>\tau(i+2)>\cdots>\tau(n)$. In other words, $\tau$ may be any permutation with $\tau(i-1)=1$, $\tau(i)=n$, and such that the entries left from 1 and right from $n$ form two decreasing sequences. There are $\binom{n-2}{i-2}$ such permutations.

In conclusion, we have the recurrence
\begin{equation*}
 F_n(123) = F_{n-1}(123) + 1 + \sum_{i=2}^{n-1}\binom{n-2}{i-2} = F_{n-1}(123) + 2^{n-2},
\end{equation*}
which implies $F_n(123) = 2^{n-1}$.

\medskip\noindent
{\sc Case} $\sigma=213$: For $n>2$, the permutation $12\cdots n$ is the only permutation in $\F_{n}(213)$ that ends with $n$, and if $\pi\in \F_{n-1}(213)$, then $1\ominus\pi \in \F_{n}(213)$.

Assume $\tau\in \F_{n}(213)$ is such that $\tau(i)=n$ for some $1<i<n$. Since $\tau$ avoids the pattern 213, we must have $\tau(1)<\tau(2)<\cdots<\tau(i-1)$ and the Fishburn condition forces $\tau(j)=j$ for every $j\in\{1,\dots,i-1\}$. Thus $\tau$ must be of the form $\tau = 1\cdots (i-1)n|\pi_R$, where $\pi_R$ may be any element of $\F_{n-i}(213)$. This implies
\[ F_n(213) = 1 + \sum_{i=1}^{n-1} F_{n-i}(213), \]
and we conclude $F_n(213) = 2^{n-1}$.

\medskip\noindent
{\sc Case} $\sigma=312$: If $\pi\in \F_{n-1}(312)$, the permutation $1\oplus\pi$ is in $\F_n(312)$. On the other hand, if $\tau$ is a permutation in $\Av_n(312)$ with $\tau(i)=1$ for some $1<i\le n$, then we must have $\tau(j)<\tau(k)$ for every $j\in\{1,\dots,i-1\}$ and $k\in\{i+1,\dots,n\}$. Moreover, the Fishburn condition forces $\tau(j)=i+1-j$ for every $j\in\{1,\dots,i-1\}$. Thus $\tau$ must be of the form $\tau = i\cdots 21|\pi_R$, where $\pi_R=\emptyset$ if $i=n$, or $\pi_R\in\F_{n-i}(312)$ if $i<n$. This implies
\[ F_n(312) = 1 + \sum_{i=1}^{n-1} F_{n-i}(312), \]
hence $F_n(312) = 2^{n-1}$.
\end{proof}

For our next result, we use a bijection between $\Av_n(321)$ and the set of Dyck paths of semilength $n$, via the left-to-right maxima.\footnote{This is a slightly different version of a bijection given by Krattenthaler \cite{Krattenthaler}.} Here a Dyck path of semilength $n$ is a simple lattice path from $(0,0)$ to $(n,n)$ that stays weakly above the diagonal $y=x$ (with vertical unit steps $U$ and horizontal unit steps $D$). On the other hand, a left-to-right maximum of a permutation $\pi$ is an element $\pi_i$ such that $\pi_j<\pi_i$ for every $j<i$. 

The bijective map between $\Av_n(321)$ and Dyck paths of semilength $n$ is defined as follows: Given $\pi\in\Av_n(321)$, write $\pi = m_1w_1m_2w_2\cdots m_sw_s$, where $m_1,\dots,m_s$ are the left-to-right maxima of $\pi$, and each $w_i$ is a subword of $\pi$. Let $|w_i|$ denote the length of $w_i$. Reading the decomposition of $\pi$ from left to right, we construct a path starting with $m_1$ $U$-steps, $|w_1|+1$ $D$-steps, and for every other subword $m_iw_i$, we add $m_i-m_{i-1}$ $U$-steps followed by $|w_i|+1$ $D$-steps. In short, identify the left-to-right maxima in the plot of $\pi$ and draw your path over them.
For example, for $\pi=351264 \in \Av_6(321)$ we get: 

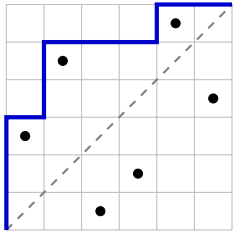
\begin{center}
\begin{tikzpicture}[scale=0.5]
\draw[lightgray] (0,0) grid (6,6); 
\draw[gray,thick, dashed] (0,0) -- (6,6);
\foreach \x/\y in {1/3,2/5,3/1,4/2,5/6,6/4}{\draw[fill] (\x-0.5,\y-0.5) circle (0.12);}
\draw[ultra thick, blue!80!black] (0,0)--(0,3)--(1,3)--(1,5)--(4,5)--(4,6)--(6,6);
\end{tikzpicture}
\end{center}
Note that $\pi=351264 \not\in\F_6$.

\begin{thm} \label{321av}
The set $\F_n(321)$ is in bijection with the set of Dyck paths of semilength $n$ that avoid the subpath $UUDU$. By \cite[Prop.~5]{STT06} we then have
\[ F_n(321) = \sum_{j=0}^{\lfloor (n-1)/2\rfloor} \frac{(-1)^j}{n-j}\binom{n-j}{j} \binom{2n-3j}{n-j+1}. \]
This is sequence \oeis{A105633}.
\end{thm}
\begin{proof}
Under the above bijection, an ascent $\pi_i<\pi_{i+1}$ in $\pi\in \Av_n(321)$ with $k=\pi_{i+1}-\pi_i$ generates the subpath $UDU^k$ in the corresponding Dyck path $P_\pi$, and if $\pi_i - 1=\pi_j$ for some $j>i+1$, then $P_\pi$ must necessarily contain the subpath $UUDU^k$. Thus we have that $\pi$ avoids the pattern $\fpattern$ if and only if $P_\pi$ avoids $UUDU$.
\end{proof}

\subsection*{Indecomposable permutations}

Let $\F_n^{\textsf{ind}}(\sigma)$ be the set of indecomposable Fishburn permutations that avoid the pattern $\sigma$, and let $\IF_n(\sigma)$ denote the number of elements in $\F_n^{\textsf{ind}}(\sigma)$. Observe that for every $\sigma$ of size $\ge 3$, we have $\IF_1(\sigma)=1$ and $\IF_2(\sigma)=1$.

We start with a fundamental known lemma, see e.g.\ \cite[Lem.~3.1]{GKZ16}.
\begin{lem} \label{lem:invert_ind}
If a pattern $\sigma$ is indecomposable, then the sequence $|\Av_n(\sigma)|$ is the {\sc invert} transform of the sequence $|\Av_n^{\mathsf{ind}}(\sigma)|$. That is, if $A^\sigma(x)$ and $A_I^\sigma(x)$ are the corresponding generating functions, then
\[   1+A^\sigma(x) = \frac{1}{1-A_I^\sigma(x)}, \text{ and so }\; A_I^\sigma(x)= \frac{A^\sigma(x)}{1+A^\sigma(x)} . \]
In particular, since $231$ is indecomposable, these identities are also valid for Fishburn permutations. The sequence $(\IF_n)_{n\in\mathbb{N}}$ that enumerates indecomposable Fishburn permutations of size $n$ starts with 
\[ 1, 1, 2, 6, 23, 104, 534, 3051, 19155, 130997, \dots. \]
\end{lem}

\begin{thm}
For $n>1$, we have $\IF_n(123) = 2^{n-1} - (n - 1)$.
\end{thm}
\begin{proof}
As discussed in the proof of Theorem~\ref{thm:123av}, for $n>2$ the set $\F_{n}(123)$ consists of elements of the form $1\ominus\pi$ with $\pi\in \F_{n-1}(123)$, and elements of the form $\tau = (i-1)\cdots 1n|\pi_R$ for some $1<i\le n$ and $\pi_R\in \F_{n-i}(123)$. This forces $\pi_R=\emptyset$ if $i=n$, and $\pi_R=(n-1)\cdots i$ if $i<n$. Thus the only decomposable elements of $\F_n(123)$ are the $n-1$ permutations
\begin{gather*}
 1n(n-1)\cdots 32, \\
 21n(n-1)\cdots 3, \\[-1ex]
 \vdots \\ 
 (n-1)\cdots 321n. 
\end{gather*}
In conclusion, $\IF_n(123) = F_n(123)-(n - 1) = 2^{n-1} - (n - 1)$.
\end{proof}

\begin{thm}
For $n>1$ and $\sigma\in\{132,213\}$, we have $\IF_n(\sigma) = 2^{n-2}$.
\end{thm}
\begin{proof}
From the proof of Theorem~\ref{thm:123av}, we know that for $n>1$ every element of $\F_n(132)$ must be of the form $1\ominus\pi$ or $\pi\oplus 1$ with $\pi\in \F_{n-1}(132)$. Since $\pi\oplus 1$ is decomposable and $1\ominus\pi$ is indecomposable, we have 
\[ \F_n^{\textsf{ind}}(132) = \{1\ominus\pi: \pi\in \F_{n-1}(132)\}. \]

We also know that $\F_{n}(213)$ consists of elements of the form $1\ominus\pi$ with $\pi\in \F_{n-1}(213)$, and elements of the form $\tau = (1\cdots (i-1)) \oplus (1\ominus \pi_R)$ for some $1<i\le n$ (with $\pi_R=\emptyset$ when $i=n$). Thus 
\[ \F_n^{\textsf{ind}}(213) = \{1\ominus\pi: \pi\in \F_{n-1}(213)\}. \]

In conclusion, if $\sigma\in\{132,213\}$, we have $\IF_n(\sigma) = F_{n-1}(\sigma) = 2^{n-2}$.
\end{proof}

Let $F^{\sigma}(x)$ and $ \IF^{\sigma}(x)$ be the generating functions associated with $(F_n(\sigma))_{n\in\mathbb{N}}$ and $(\IF_n(\sigma))_{n\in\mathbb{N}}$, respectively.
\begin{thm}
For $\sigma\in\{231,312,321\}$, we have 
\[ \IF^{\sigma}(x) = \frac{F^{\sigma}(x)}{1+F^{\sigma}(x)}. \]
In particular, $\IF_n(231) = C_{n-1} $ and $\IF_n(312) = 1$.
\end{thm}
\begin{proof}
This follows from \eqref{eq:p231}, Theorem~\ref{thm:123av}, and Lemma~\ref{lem:invert_ind}.
\end{proof}

\begin{thm}
The enumerating sequence $a_n = \IF_n(321)$ satisfies the recurrence relation
\begin{equation*} 
a_n = a_{n-1} + \sum_{j=2}^{n-2} a_ja_{n-j} \text{ for } n\ge 4,
\end{equation*}
with $a_1=a_2=a_3=1$. This is sequence \oeis{A082582}.
\end{thm}
\begin{proof}
We use the same Dyck path approach as in the proof of Theorem~\ref{321av}. Under this bijection, indecomposable permutations correspond to Dyck paths that do not touch the line $y=x$ except at the end points. 

Let $\mathcal{A}_n$ be the set of Dyck paths corresponding to $\F_n^{\textsf{ind}}(321)$.  We will prove that $a_n=|\mathcal{A}_n|$ satisfies the claimed recurrence relation. Clearly, for $n=1,2,3$, the only indecomposable Fishburn permutations are $1$, $21$, and $312$, which correspond to the Dyck paths $UD$, $U^2D^2$, and $U^3D^3$, respectively. Thus $a_1=a_2=a_3=1$.

Note that indecomposable permutations may not start with $1$ or end with $n$. Moreover, every element of $\pi\in\F_n^{\textsf{ind}}(321)$ must be of the form $m1\pi(3)\cdots\pi(n)$ with $m\ge 3$. Therefore, the elements of $\mathcal{A}_n$ have no peaks at the points $(0,1)$, $(0,2)$, or $(n-1,n)$, and for $n>3$ their first return to the line $y=x+1$ happens at a lattice point $(x,x+1)$ with $x\in [2,n-1]$. 

Dyck paths in $\mathcal{A}_n$ having $(n-1,n)$ as their first return to $y=x+1$, are in one-to-one correspondence with the elements of $\mathcal{A}_{n-1}$ (just remove the first $U$ and the last $D$ of the longer path). Now, for $j\in\{2,\dots,n-2\}$, the set of paths $P\in \mathcal{A}_n$ having first return to $y=x+1$ at the point $(j,j+1)$ corresponds uniquely to the set of all pairs $(P_j,P_{n-j})$ with $P_j\in \mathcal{A}_{j}$ and $P_j\in \mathcal{A}_{n-j}$. For example,

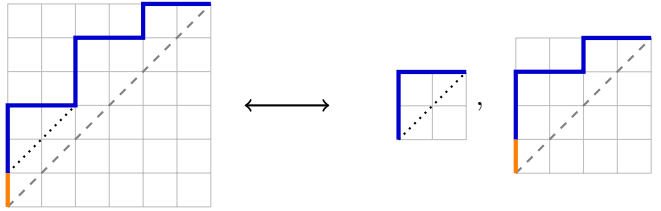
\begin{center}
\begin{tikzpicture}[scale=0.45]
\begin{scope}
\draw[lightgray] (0,0) grid (6,6); 
\draw[gray,thick, dashed] (0,0) -- (6,6);
\draw[thick, dotted] (0,1) -- (2,3);
\draw[ultra thick, blue!80!black] (0,1)--(0,3)--(2,3)--(2,5)--(4,5)--(4,6)--(6,6);
\draw[ultra thick, orange] (0,0) -- (0,1);
\draw[thick,<->] (7,3) -- (9.5,3);
\end{scope}
\begin{scope}[xshift=30em,yshift=12ex]
\draw[lightgray] (0,0) grid (2,2); 
\draw[thick, dotted] (0,0) -- (2,2);
\draw[ultra thick, blue!80!black] (0,0)--(0,2)--(2,2);
\node at (2.4,1) {,};
\end{scope}
\begin{scope}[xshift=39em,yshift=6ex]
\draw[lightgray] (0,0) grid (4,4); 
\draw[gray,thick, dashed] (0,0) -- (4,4);
\draw[ultra thick, orange] (0,0)--(0,1);
\draw[ultra thick, blue!80!black] (0,1)--(0,3)--(2,3)--(2,4)--(4,4);
\node at (4.4,2) {.};
\end{scope}
\end{tikzpicture}
\end{center}
This implies that there are $a_ja_{n-j}$ paths in $\mathcal{A}_n$ having the point $(j,j+1)$ as their first return to the line $y=x+1$. Finally, summing over $j$ gives the claimed identity.
\end{proof}

\noindent
Here is a summary of our enumeration results for patterns of size 3:

\begin{table}[ht]
\small
\begin{tabular}{|c|l|c|} \hline
\R Pattern $\sigma$ & \hspace{42pt} $|\F_n(\sigma)|$ & OEIS  \\[2pt] \hline
\R \parbox{8em}{123, 132, 213, 312} & 1, 2, 4, 8, 16, 32, 64, 128, 256, 512, \dots & A000079 \\ \hline
\R 231 & 1, 2, 5, 14, 42, 132, 429, 1430, 4862, \dots & A000108 \\ \hline
\R 321 & 1, 2, 4, 9, 22, 57, 154, 429, 1223, 3550, \dots & A105633 \\ \hline
\end{tabular}
\bigskip
\caption{$\sigma$-avoiding Fishburn permutations.}
\label{tab:3Patterns}
\end{table}

\begin{table}[ht]
\small
\begin{tabular}{|c|l|c|} \hline
\R Pattern $\sigma$ & \hspace{42pt} $|\F_n^{\textsf{ind}}(\sigma)|$ & OEIS \\[2pt] \hline
\R 123 & 1, 1, 2, 5, 12, 27, 58, 121, 248, 503, \dots & A000325 \\ \hline
\R 132, 213 & 1, 1, 2, 4, 8, 16, 32, 64, 128, 256, \dots & A000079 \\ \hline
\R 231 & 1, 1, 2, 5, 14, 42, 132, 429, 1430, 4862, \dots & A000108 \\ \hline
\R 312 & 1, 1, 1, 1, 1, 1, 1, 1, 1, 1, \dots & A000012 \\ \hline
\R 321 & 1, 1, 1, 2, 5, 13, 35, 97, 275, 794, \dots & A082582 \\ \hline
\end{tabular}
\bigskip
\caption{$\sigma$-avoiding indecomposable Fishburn permutations.}
\label{tab:3PatternsInd}
\end{table}

%%%%%%%%%%%%%%%%%%%%%%%%%%%%%%%%%%%%%%%%%%%%
\section{Avoiding patterns of size 4}
\label{sec:length4patterns}

In this section, we discuss the enumeration of Fishburn permutations that avoid a pattern of size 4. There are at least 13 Wilf equivalence classes that we break down into three categories: 10 classes with a single pattern, 2 classes with (conjecturally) three patterns each, and a larger class with eight patterns enumerated by the Catalan numbers. 

\begin{table}[ht]
\small
\begin{tabular}{|c|c|c|} \hline
\R Pattern $\sigma$ & $|\F_n(\sigma)|$ & OEIS \\[2pt] \hline
\R 1342 & 1, 2, 5, 15, 51, 188, 731, 2950, \dots & A007317 \\ \hline
\R 1432 & 1, 2, 5, 14, 43, 142, 495, 1796, \dots & \\ \hline
\R 2314 & 1, 2, 5, 15, 52, 200, 827, 3601, \dots & \\ \hline
\R 2341 & 1, 2, 5, 15, 52, 202, 858, 3910, \dots & \\ \hline
\R 3412 & 1, 2, 5, 15, 52, 201, 843, 3764, \dots & A202062(?) \\ \hline
\R 3421 & 1, 2, 5, 15, 52, 203, 874, 4076, \dots & \\ \hline
\R 4123 & 1, 2, 5, 14, 42, 133, 442, 1535, \dots & \\ \hline
\R 4231 & 1, 2, 5, 15, 52, 201, 843, 3765, \dots & \\ \hline
\R 4312 & 1, 2, 5, 14, 43, 143, 508, 1905, \dots & \\ \hline
\R 4321 & 1, 2, 5, 14, 45, 162, 639, 2713, \dots & \\ \hline
\end{tabular}
\bigskip
\caption{Equivalence classes with a single pattern}
\label{tab:4Patterns}
\end{table}

We will provide a proof for the enumeration of the class $\F_n(1342)$, but our main focus in this paper will be on the enumeration of the equivalence class given in Table~\ref{tab:CatalanClass}.

For the remaining patterns we have the following conjectures. 
\begin{conjecture}
$\F_n(2413)\sim \F_n(2431)\sim \F_n(3241)$.
\end{conjecture}

\begin{conjecture}
$\F_n(3214)\sim \F_n(4132)\sim \F_n(4213)$.
\end{conjecture}

Our first result of this section involves the binomial transform of the Catalan numbers, namely the sequence \oeis{A007317}.
\begin{thm}
\[ F_n(1342) = \sum_{k=1}^n \binom{n-1}{k-1} C_{n-k}. \]
\end{thm}
\begin{proof}
Let $\mathcal{A}_{n,k}$ be the set of all permutations $\pi \in S_n$ such that
\begin{itemize}
\item[$\circ$] $\pi(k)=1$ and $\pi(1)>\pi(2)>\cdots>\pi(k-1)$,
\item[$\circ$] $\pi(k+1)\cdots\pi(n) \in \Av_{n-k}(231)$,
\end{itemize}
and let $\A_n = \bigcup_{k=1}^n \mathcal{A}_{n,k}$. Clearly, 
\[ |\A_n| = \sum_{k=1}^n |\mathcal{A}_{n,k}| = \sum_{k=1}^n\binom{n-1}{k-1} C_{n-k}. \]
We will prove the theorem by showing that $\A_n = \F_n(1342)$.

First of all, since $\mathcal{A}_{n,k} \subset \Av_n(1342)$ and $\Av_{n-k}(231) = \F_{n-k}(231)$ for every $n$ and $k$, we have $\A_n\subset  \F_n(1342)$.

Going in the other direction, let $\pi \in \F_n(1342)$ and let $k$ be such that $\pi(k)=1$. Thus $\pi$ is of the form $\pi = \pi(1)\cdots\pi(k-1) \,1\, \pi(k+1)\cdots \pi(n)$, which implies $\pi(k+1)\cdots \pi(n) \in \Av_{n-k}(231)$. Now, if there is a $j\in\{1,\dots,k-2\}$ such that 
\[ \pi(1)>\cdots > \pi(j)<\pi(j+1),\]
then $\pi(j)-1=\pi(\ell)$ for some $\ell>j+1$, and the pattern $\pi(j)\pi(j+1)\pi(\ell)$ would violate the Fishburn condition. In other words, the entries left from $\pi(k)=1$ must form a decreasing sequence, which implies $\pi\in\mathcal{A}_{n,k}\subset \A_n$. Thus $\F_n(1342)\subset\A_n$, and we conclude that $\A_n = \F_n(1342)$. 
\end{proof}

\subsection*{Catalan equivalence class}

\begin{table}[ht]
\small
\begin{tabular}{|c|c|c|} \hline
\R Pattern $\sigma$ & $|\F_n(\sigma)|$ & OEIS \\[2pt] \hline
\rule[-2.5ex]{0ex}{6.5ex} \parbox{10.5em}{1234, 1243, 1324, 1423,\\ 2134, 2143, 3124, 3142} 
& 1, 2, 5, 14, 42, 132, 429, 1430, 4862, \dots & A000108 \\ \hline
\end{tabular}
\bigskip
\caption{} 
\label{tab:CatalanClass}
\end{table}

The remaining part of this section is devoted to prove that  $|\F_n(\sigma)|=C_n$ for every $\sigma\in\{1234, 1243, 1324, 1423, 2134, 2143, 3124, 3142\}$.

\begin{thm} \label{thm:Catalan3142}
We have $\F_n(3142)=\F_n(231)$, hence $F_n(3142) = C_n$.
\end{thm}
\begin{proof}
Since 3142 contains the pattern 231, we have $\F_n(231)\subseteq \F_n(3142)$. 

To prove the reverse inclusion, suppose there exists $\pi\in \F_n(3142)$ such that $\pi$ contains the pattern 231. Let $i<j<k$ be the positions of the most-left closest 231 pattern contained in $\pi$. By that we mean:
\begin{itemize}
\item[$\circ$] $\pi(k)<\pi(i)<\pi(j)$,
\item[$\circ$] $\pi(i)$ is the most-left entry of $\pi$ involved in a 231 pattern,
\item[$\circ$] $\pi(j)$ is the first entry with $j>i$ such that $\pi(i)<\pi(j)$,
\item[$\circ$] $\pi(k)$ is the largest entry with $k>j$ such that $\pi(k)<\pi(i)$.
\end{itemize}
In other words, assume the plot of $\pi$ is of the form

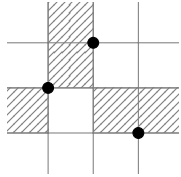
\begin{center}
\begin{tikzpicture}[scale=0.6]
\clip (0.1,0.1) rectangle (3.9,3.9);
\draw[gray] (0,0) grid (4,4);
\draw[pattern=north east lines, pattern color=gray, draw=gray] (1,2) rectangle (2,4);
\draw[pattern=north east lines, pattern color=gray, draw=gray] (0,1) rectangle (1,2);
\draw[pattern=north east lines, pattern color=gray, draw=gray] (2,1) rectangle (4,2);
\foreach \x/\y in {1/2,2/3,3/1}{\draw[fill] (\x,\y) circle (0.12);}
\end{tikzpicture}
\end{center}
where no elements of $\pi$ may occur in the shaded regions. It follows that, if $\ell$ is the position of $\pi(k)+1$, then $i\le \ell<j$. But this is not possible since, $\pi(\ell)<\pi(\ell+1)$ violates the Fishburn condition, and $\pi(\ell)>\pi(\ell+1)$ implies $\pi(k)>\pi(\ell+1)$ which forces the existence of a 3142 pattern. In conclusion, no permutation $\pi\in\F_n(3142)$ is allowed to contain a 231 pattern. Therefore, $\F_n(3142)\subseteq \F_n(231)$ and we obtain the claimed equality.
\end{proof}

\begin{thm} \label{thm:WestAlgorithm}
$\F_n(1234)\sim \F_n(1243)$ and $\F_n(2134)\sim \F_n(2143)$.
\end{thm}
\begin{proof}
In order to prove both Wilf equivalence relations, we use a bijection 
\[ \phi:\Av_n(\tau\oplus 12)\to \Av_n(\tau\oplus 21) \] 
given by West in \cite{WestThesis}, which we proceed to describe.

For $\pi \in S_n$ and $\sigma\in S_k$, $k<n$, let $B_{\pi}(\sigma)$ be the set of maximal values of all instance of the pattern $\sigma$ in the permutation $\pi$. For example, $B_\pi(\sigma)=\emptyset$ if $\pi$ avoids $\sigma$, and for $\pi=531968274$ we have $B_\pi(123)=\{4,7,8\}$.

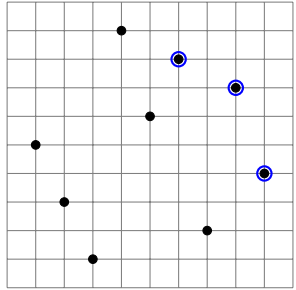
\begin{figure}[ht]
\begin{tikzpicture}[scale=0.38]
\draw[gray] (0,0) grid (10,10);
\foreach \x/\y in {1/5,2/3,3/1,4/9,5/6,6/8,7/2,8/7,9/4}{\draw[fill] (\x,\y) circle (0.15);}
\foreach \x/\y in {6/8,8/7,9/4}{\draw[thick,blue] (\x,\y) circle (0.25);}
\end{tikzpicture}
\caption{$\pi=531968274$ and $B_\pi(123)=\{4,7,8\}$}
\end{figure}

For $\pi\in \Av_n(\tau\oplus 12)$, let $\ell$ be the number of elements in $B_\pi(\tau\oplus 1)$. If $\ell=0$, we define $\phi(\pi) = \pi$. If $\ell>0$, we let $i_1<\dots<i_\ell$ be the positions in $\pi$ of the elements of $B_\pi(\tau\oplus 1)$ and define 
\[  \tilde\pi(j) = \pi(j) \;\text{ if } j\not\in\{i_1,\dots,i_\ell\}. \]
Note that $\pi\in \Av_n(\tau\oplus 12)$ implies $\pi(i_1)>\cdots>\pi(i_\ell)$. 

Let $b_1$ be the smallest element of $B_\pi(\tau\oplus 1)$ such that $\tilde\pi(1)\cdots \tilde\pi(i_1-1)b_1$ contains the pattern $\tau\oplus 1$. Define
\[  \tilde\pi(i_1) = b_1. \]
Iteratively, for $k=2,\dots,\ell$, we let $b_k$ be the smallest element of $B_\pi(\tau\oplus 1)\backslash\{b_1,\dots,b_{k-1}\}$ such that $\tilde\pi(1)\cdots \tilde\pi(i_k-1)b_k$ contains the pattern $\tau\oplus 1$. We then define
\[  \tilde\pi(i_k) = b_k \;\text{ for } k=2,\dots,\ell, \]
to complete the definition of $\tilde\pi = \phi(\pi)$.

For example, for $\pi=53196{\color{blue}\bf 8}2{\color{blue}\bf 74}$ and $\tau=12$, we have $\ell=3$ and $\tilde\pi = 53196{\color{orange}\bf 7}2{\color{orange}\bf 48}$.

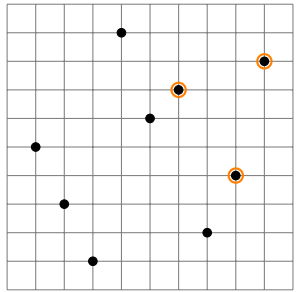
\begin{figure}[ht]
\begin{tikzpicture}[scale=0.38]
\draw[gray] (0,0) grid (10,10);
\foreach \x/\y in {1/5,2/3,3/1,4/9,5/6,6/7,7/2,8/4,9/8}{\draw[fill] (\x,\y) circle (0.15);}
\foreach \x/\y in {6/7,8/4,9/8}{\draw[thick,orange] (\x,\y) circle (0.25);}
\end{tikzpicture}
\caption{$\tilde\pi = \phi(531968274) = 531967248$}
\end{figure}

It is easy to check that the map $\phi$ induces a bijection
\[ \phi:\F_n(\tau\oplus 12)\to \F_n(\tau\oplus 21). \] 
Indeed, if $\pi(i_k)\in B_\pi(\tau\oplus 1)$ is such that $\pi(i_k-1)<\pi(i_k)$, then $\tilde\pi(i_k-1)=\pi(i_k-1)$ and the pair $\tilde\pi(i_k-1), \tilde\pi(i_k)$ does not create a pattern $\fpattern$.

On the other hand, if $\pi(i_1)>\dots>\pi(i_k)$ is a maximal descent of elements from $B_\pi(\tau\oplus 1)$, and if $\pi(i_j)-1>0$ (for $j\in\{1,\dots,k\}$) is not part of that descent, then $\pi(i_j)-1$ must be left from $\pi(i_1)$ and so any ascent in $\tilde\pi(i_1)\cdots\tilde\pi(i_k)$ cannot create the pattern $\fpattern$.

Thus, if $\pi\in \Av_n(\tau\oplus 12)$ is Fishburn, so is $\tilde\pi = \phi(\pi)\in \Av_n(\tau\oplus 21)$.
\end{proof}

\begin{thm} \label{thm:1423-1243}
$\F_n(1423)\sim \F_n(1243)\sim \F_n(1234)\sim \F_n(1324)$.
\end{thm}
\begin{proof}
Let $\alpha:\F_n(1423)\to \F_n(1243)$ be the map defined through the following algorithm.

\medskip\noindent
{\em Algorithm} $\alpha$: Let $\pi \in \F_n(1423)$ and set $\tilde\pi=\pi$.
\begin{enumerate}[Step 1:]
\item If $\tilde\pi \not\in \Av_n(1243)$, let $i<j<k<\ell$ be the positions of the most-left 1243 pattern contained in $\tilde\pi$. Redefine $\tilde\pi$ by moving $\tilde\pi(k)$ to position $j$, shifting the entries at positions $j$ through $k-1$ one step to the right.
\item If $\tilde\pi \in \Av_n(1243)$, then return $\alpha(\pi)=\tilde\pi$; otherwise go to Step 1.
\end{enumerate}

\medskip
For example, for $\pi = 2135476 \in \F(1423)$, the above algorithm yields
\begin{align*}
 {\bf 2}1{\bf 354}76 &\longrightarrow 21{\bf 5}3476 \not\in \Av(1243)\\
 &\swarrow \\
 {\bf 2}1{\bf 5} 34 {\bf 76} &\longrightarrow 21{\bf 7}5346 \in \Av(1243)
\end{align*}
and so $\alpha(\pi) = 2175346\in \F(1243)$.

Observe that the map $\alpha$ changes every 1243 pattern into a 1423 pattern. To see that it preserves the Fishburn condition, let $\pi \in \F_n(1423)$ be such that $\pi(i)$, $\pi(j)$, $\pi(k)$, $\pi(\ell)$ form a most-left 1243 pattern. Thus, at first, $\pi$ must be of the form

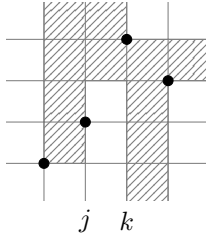
\begin{center}
\begin{tikzpicture}[scale=0.55]
\begin{scope}
\clip (0.1,0.1) rectangle (4.9,4.9);
\draw[gray] (0,0) grid (5,5);
\draw[pattern=north east lines, pattern color=gray, draw=gray] 
  (1,1) -- (2,1) -- (2,3) -- (3,3) -- (3,0) -- (4,0) -- (4,3) -- (5,3) -- (5,4) -- (3,4) -- (3,5) -- (1,5) -- cycle ;
\foreach \x/\y in {1/1,2/2,3/4,4/3}{\draw[fill] (\x,\y) circle (0.12);}
\end{scope}
\node at (2,-0.4) {\small $j$};
\node at (3,-0.4) {\small $k$};
\end{tikzpicture}
\end{center}
where no elements of $\pi$ may occur in the shaded regions. In particular, we must have
\begin{equation} \label{eq:ascents}
 \pi(j-1)<\pi(j) \;\text{ and }\; \pi(k-1)<\pi(k).
\end{equation}
Hence the step of moving $\pi(k)$ to position $j$ does not create a new ascent, and therefore it cannot create a pattern $\fpattern$. After one iteration, $\tilde\pi$ takes the form

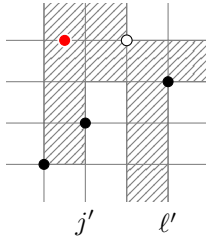
\begin{center}
\begin{tikzpicture}[scale=0.55]
\begin{scope}
\clip (0.1,0.1) rectangle (4.9,4.9);
\draw[gray] (0,0) grid (5,5);
\draw[pattern=north east lines, pattern color=gray, draw=gray] 
  (1,1) -- (2,1) -- (2,3) -- (3,3) -- (3,0) -- (4,0) -- (4,3) -- (5,3) -- (5,4) -- (3,4) -- (3,5) -- (1,5) -- cycle ;
\foreach \x/\y in {1/1,2/2,4/3}{\draw[fill] (\x,\y) circle (0.12);}
\draw[fill,white, draw=black] (3,4) circle (0.13);
\draw[white,thick] (1.5,4) circle (0.15);
\draw[fill,red] (1.5,4) circle (0.12);
\end{scope}
\node at (2,-0.4) {\small $j'$};
\node at (4,-0.4) {\small $\ell'$};
\end{tikzpicture}
\end{center}
and if the most-left 1243 pattern $\tilde\pi(i)$, $\tilde\pi(j)$, $\tilde\pi(k)$, $\tilde\pi(\ell)$ contained in $\tilde\pi$  has its second entry at a position different from $j'$, then $\tilde\pi$ must satisfy \eqref{eq:ascents} and no $\fpattern$ will be created. 

Otherwise, if $j=j'$, then either $k=\ell'$ or $\ell=\ell'$. In the first case, we have $\tilde\pi(k-1)<\tilde\pi(k)$ and $\tilde\pi(k) < \tilde\pi(j-1)$, so moving $\tilde\pi(k)$ to position $j$ does not create a new ascent. On the other hand, if $\ell=\ell'$, then $\tilde\pi(k)>\tilde\pi(j-1)$ but $\tilde\pi(j-1)-1$ must be left from $\tilde\pi(i)$. Therefore, also in this case, applying an iteration of $\alpha$ will preserve the Fishburn condition. 

\smallskip
We conclude that, if $\pi$ is Fishburn, so is $\alpha(\pi)$.

\medskip
The reverse map $\beta:\F_n(1243)\to \F_n(1423)$ is given by the following algorithm.

\medskip\noindent
{\em Algorithm} $\beta$:
Let $\tau \in \F_n(1243)$ and set $\tilde\tau=\tau$.
\begin{enumerate}[\quad Step 1:]
\item If $\tilde\tau \not\in \Av_n(1423)$, let $i<j<k<\ell$ be the positions of the most-right 1423 pattern contained in $\tilde\tau$. Redefine $\tilde\tau$ by moving $\tilde\tau(j)$ to position $k$, shifting the entries at positions $j+1$ through $k$ one step to the left.
\item If $\tilde\tau \in \Av_n(1423)$, then return $\beta(\tau)=\tilde\tau$; otherwise go to Step 1.
\end{enumerate}

\medskip
In conclusion, the map $\alpha$ gives a bijection $\F_n(1423)\to \F_n(1243)$.

With a similar argument, it can be verified that $\alpha$ also maps $\F_n(1324)\to \F_n(1234)$ bijectively. Finally, the equivalence $\F_n(1243)\sim \F_n(1234)$ was shown in Theorem~\ref{thm:WestAlgorithm}.
\end{proof}

\begin{thm}
$\F_n(3142)\sim \F_n(3124)\sim \F_n(1324)$.
\end{thm}
\begin{proof}
We will define two maps 
\[  \F_n(3142) \stackrel{\alpha_1}{\longrightarrow} \F_n(3124) \;\text{ and }\; 
    \F_n(3124) \stackrel{\alpha_2}{\longrightarrow} \F_n(1324) \]
through algorithms similar to the one introduced in the proof of Theorem~\ref{thm:1423-1243}.

\medskip\noindent
{\em Algorithm} $\alpha_1$:
Let $\pi\in \F_n(3142)$ and set $\tilde\pi=\pi$.
\begin{enumerate}[\quad Step 1:]
\item If $\tilde\pi \not\in \Av_n(3124)$, let $i<j<k<\ell$ be the positions of the most-left 3124 pattern contained in $\tilde\pi$. Redefine $\tilde\pi$ by moving $\tilde\pi(\ell)$ to position $k$, shifting the entries at positions $k$ through $\ell-1$ one step to the right.
\item If $\tilde\pi \in \Av_n(3124)$, then return $\alpha_1(\pi)=\tilde\pi$; otherwise go to Step 1.
\end{enumerate}

As $\alpha$ in Theorem~\ref{thm:1423-1243}, the map $\alpha_1$ is reversible and preserves the Fishburn condition. For an illustration of the latter claim, here is a sketch of a permutation $\pi\in \F_n(3142)$ having a most-left 3124 pattern, together with the sketch of $\tilde\pi$ after one iteration of $\alpha_1$:

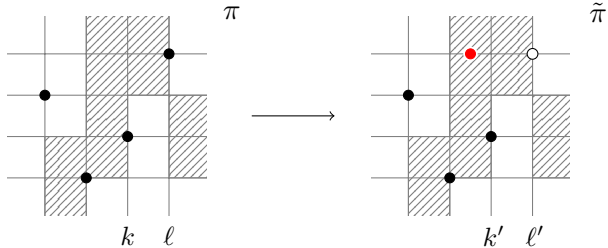
\begin{center}
\begin{tikzpicture}[scale=0.55]
\begin{scope}
\clip (0.1,0.1) rectangle (4.9,4.9);
\draw[gray] (0,0) grid (5,5);
\draw[pattern=north east lines, pattern color=gray, draw=gray] (4,1) rectangle (5,3);
\draw[pattern=north east lines, pattern color=gray, draw=gray] 
 (1,0) -- (2,0) -- (2,1) -- (3,1) -- (3,3) -- (4,3) -- (4,5) -- (2,5) -- (2,2) -- (1,2) -- cycle;
\foreach \x/\y in {1/3,2/1,3/2,4/4}{\draw[fill] (\x,\y) circle (0.12);}
\end{scope}
\node at (3,-0.4) {\small $k$};
\node at (4,-0.4) {\small $\ell$};
\node[right=1pt] at (5,5) {$\pi$};
\draw[->] (6,2.5) -- (8,2.5);
\begin{scope}[xshift=250]
\clip (0.1,0.1) rectangle (4.9,4.9);
\draw[gray] (0,0) grid (5,5);
\draw[pattern=north east lines, pattern color=gray, draw=gray] (4,1) rectangle (5,3);
\draw[pattern=north east lines, pattern color=gray, draw=gray] 
 (1,0) -- (2,0) -- (2,1) -- (3,1) -- (3,3) -- (4,3) -- (4,5) -- (2,5) -- (2,2) -- (1,2) -- cycle;
\foreach \x/\y in {1/3,2/1,3/2}{\draw[fill] (\x,\y) circle (0.12);}
\draw[fill,white, draw=black] (4,4) circle (0.13);
\draw[white,thick] (2.5,4) circle (0.15);
\draw[fill,red] (2.5,4) circle (0.12);
\end{scope}
\node at (11.85,-0.4) {\small $k'$};
\node at (12.85,-0.4) {\small $\ell'$};
\node[right=1pt] at (13.85,5) {$\tilde\pi$};
\end{tikzpicture}
\end{center}
where no elements of the permutation $\pi$ may occur in the shaded regions. 

Since $\pi(k-1)<\pi(k)$ and $\pi(\ell-1)<\pi(\ell)$, the Fishburn condition of $\pi$ is preserved after the first iteration of $\alpha_1$. Further, if $\tilde\pi$ has a most-left 3124 pattern with third entry at position $k'$, then we must have $\ell>\ell'$ and $\tilde\pi(\ell)>\tilde\pi(i)$. If $\tilde\pi(\ell)<\tilde\pi(k'-1)$, no new ascent can be created when moving $\tilde\pi(\ell)$ to position $k'$. Otherwise, if $\tilde\pi(\ell)>\tilde\pi(k'-1)$, then either both entries were part of an ascent in the original $\pi$ or every entry between $\tilde\pi(\ell')$ and $\tilde\pi(\ell)$ must be smaller than $\tilde\pi(j)$. Since $\pi\in \Av_n(3142)$, the latter would imply that $\tilde\pi(k'-1)-1$ is left from $\tilde\pi(j)$. In any case, no pattern $\fpattern$ will be created in the next iteration of $\alpha_1$.

Since any later iteration of $\alpha_1$ may essentially be reduced to one of the above cases, we conclude that $\alpha_1$ preserves the Fishburn condition.

\medskip\noindent
{\em Algorithm} $\alpha_2$:
Let $\pi \in \F_n(3124)$ and set $\tilde\pi=\pi$.
\begin{enumerate}[\quad Step 1:]
\item If $\tilde\pi \not\in \Av_n(1324)$, let $i<j<k<\ell$ be the positions of the most-left 1324 pattern contained in $\tilde\pi$. Redefine $\tilde\pi$ by moving $\tilde\pi(j)$ to position $i$, shifting the entries at positions $i$ through $j-1$ one step to the right.
\item If $\tilde\pi \in \Av_n(1324)$, then return $\alpha_2(\pi)=\tilde\pi$; otherwise go to Step 1.
\end{enumerate}

This map is reversible and preserves the Fishburn condition. As before, we will illustrate the Fishburn property by sketching the plot of a permutation $\pi\in \F_n(3124)$ that contains a most-left 1324 pattern $\pi(i)$, $\pi(j)$, $\pi(k)$, $\pi(\ell)$, together with the sketch of the permutation $\tilde\pi$ obtained after one iteration of $\alpha_2$:

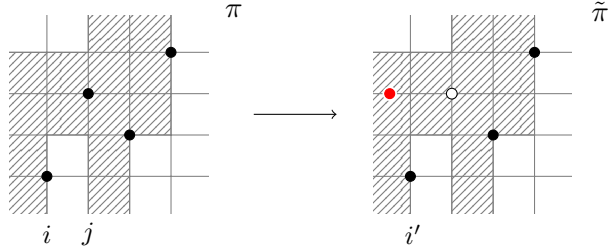
\begin{center}
\begin{tikzpicture}[scale=0.55]
\begin{scope}
\clip (0.1,0.1) rectangle (4.9,4.9);
\draw[gray] (0,0) grid (5,5);
\draw[pattern=north east lines, pattern color=gray, draw=gray] 
 (0,0)--(1,0)--(1,2)--(2,2)--(2,0)--(3,0)--(3,2)--(4,2)--(4,5)--(2,5)--(2,4)--(0,4)-- cycle;
\foreach \x/\y in {1/1,2/3,3/2,4/4}{\draw[fill] (\x,\y) circle (0.12);}
\end{scope}
\node at (1,-0.4) {\small $i$};
\node at (2,-0.4) {\small $j$};
\node[right=1pt] at (5,5) {$\pi$};
\draw[->] (6,2.5) -- (8,2.5);
\begin{scope}[xshift=250]
\clip (0.1,0.1) rectangle (4.9,4.9);
\draw[gray] (0,0) grid (5,5);
\draw[pattern=north east lines, pattern color=gray, draw=gray] 
 (0,0)--(1,0)--(1,2)--(2,2)--(2,0)--(3,0)--(3,2)--(4,2)--(4,5)--(2,5)--(2,4)--(0,4)-- cycle;
\foreach \x/\y in {1/1,3/2,4/4}{\draw[fill] (\x,\y) circle (0.12);}
\draw[fill,white, draw=black] (2,3) circle (0.13);
\draw[white,thick] (0.5,3) circle (0.15);
\draw[fill,red] (0.5,3) circle (0.12);
\end{scope}
\node at (9.85,-0.4) {\small $i'$};
\node[right=1pt] at (13.85,5) {$\tilde\pi$};
\end{tikzpicture}
\end{center}
Since no elements of the permutation $\pi$ may occur in the shaded regions, we must have either $i=1$ or $\pi(i-1)>\pi(j)$. Consequently, moving $\pi(j)$ to position $i$ will not create a new ascent and the Fishburn condition will be preserved.

Similarly, if $\tilde\pi$ has a most-left 1324 pattern with first entry at a position different from $i'$, or if $\tilde\pi(i)=\tilde\pi(i')$ and $\tilde\pi(j)<\tilde\pi(i'-1)$, then no new ascent will be created and the next $\tilde\pi$ will be Fishburn. It is not possible to have $\tilde\pi(i)=\tilde\pi(i')$ and $\tilde\pi(j)>\tilde\pi(i'-1)$.

In summary, $\alpha_1$ and $\alpha_2$ are both bijective maps.
\end{proof}

The following theorem completes the enumeration of the Catalan class (see Table~\ref{tab:CatalanClass}).

\begin{thm}
$\F_n(3142)\sim \F_n(2143)$.
\end{thm}
\begin{proof}
Let $\gamma:\F_n(3142)\to \F_n(2143)$ be the map defined through the following algorithm.

\medskip\noindent
{\em Algorithm} $\gamma$: Let $\pi \in \F_n(3142)$ and set $\tilde\pi=\pi$.
\begin{enumerate}[\quad Step 1:]
\item If $\tilde\pi \not\in \Av_n(2143)$, let $i<j<k$ be the positions of the most-left 213 pattern contained in $\tilde\pi$ such that $\tilde\pi(i)$, $\tilde\pi(j)$, $\tilde\pi(k)$, $\tilde\pi(\ell)$ form a 2143 pattern for some $\ell>k$. Let $\ell_m$ be the position of the smallest such $\tilde\pi(\ell)$, and let
\[ Q=\{q\in [n]: \tilde\pi(i)\le \tilde\pi(q)<\tilde\pi(\ell_m)\}. \]
Redefine $\tilde\pi$ by replacing $\tilde\pi(\ell_m)$ with $\tilde\pi(i)$, adding 1 to $\tilde\pi(q)$ for every $q\in Q$.
\item If $\tilde\pi \in \Av_n(2143)$, then return $\gamma(\pi)=\tilde\pi$; otherwise go to Step 1.
\end{enumerate}

For example, if $\pi=4312576$, then $\gamma(\pi) = 5412673$ (after 2 iterations, see Figure~\ref{fig:gamma}).

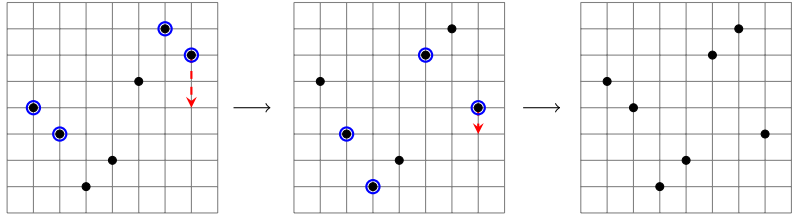
\begin{figure}[ht]
\begin{tikzpicture}[scale=0.35]
\begin{scope}
\draw[gray] (0,0) grid (8,8);
\draw[dashed,thick,red,-stealth] (7,6) -- (7,4);
\foreach \x/\y in {1/4,2/3,3/1,4/2,5/5,6/7,7/6}{\draw[fill] (\x,\y) circle (0.15);}
\foreach \x/\y in {1/4,2/3,6/7,7/6}{\draw[thick,blue] (\x,\y) circle (0.25);}
\end{scope}
\draw[->] (8.6,4) -- (10,4);
\begin{scope}[xshift=310]
\draw[gray] (0,0) grid (8,8);
\draw[dashed,thick,red,-stealth] (7,4) -- (7,3);
\foreach \x/\y in {1/5,2/3,3/1,4/2,5/6,6/7,7/4}{\draw[fill] (\x,\y) circle (0.15);}
\foreach \x/\y in {2/3,3/1,5/6,7/4}{\draw[thick,blue] (\x,\y) circle (0.25);}
\end{scope}
\draw[->] (19.6,4) -- (21,4);
\begin{scope}[xshift=620]
\draw[gray] (0,0) grid (8,8);
\foreach \x/\y in {1/5,2/4,3/1,4/2,5/6,6/7,7/3}{\draw[fill] (\x,\y) circle (0.15);}
\end{scope}
\end{tikzpicture}
\caption{Algorithm $\gamma$: $4312576 \to 5312674 \to 5412673$.}
\label{fig:gamma}
\end{figure}

The map $\gamma$ is reversible. Moreover, observe that:
\begin{enumerate}[\quad (a)]
 \item since $\tilde\pi(\ell_m)$ is the smallest entry such that $\tilde\pi(i)<\tilde\pi(\ell_m)<\tilde\pi(k)$, replacing $\tilde\pi(\ell_m)$ with $\tilde\pi(i)$ (which is equivalent to moving the plot of $\tilde\pi(\ell_m)$ down to height $\tilde\pi(i)$) will not create any new ascent at position $\ell_m$;
 \item since $\tilde\pi(i)$ is chosen to be the first entry of a most-left 2143 pattern, $\tilde\pi(i)-1$ must be right from $\tilde\pi(i)$. Hence, replacing $\tilde\pi(i)$ by $\tilde\pi(i)+1$ cannot create a new pattern $\fpattern$.
\end{enumerate}

In conclusion, $\gamma$ preserves the Fishburn condition and gives the claimed bijection.
\end{proof}

%%%%%%%%%%%%%%%%%%%%%%%%%%%%%%%%%%%%%%%%%%%%
\section{Further remarks}
\label{sec:remarks}

\begin{table}[ht]
\small
\def\R{\rule[-1ex]{0ex}{3.6ex}}
\begin{tabular}{|c|l|c|c|} \hline
\R Pattern $\sigma$ & \hspace{42pt} $|\F_n^{\textsf{ind}}(\sigma)|$ & OEIS \\[2pt] \hline
\R 1234 & 1, 1, 2, 6, 22, 85, 324, 1204,\dots & \\ \hline
\R 1243, 2134 & 1, 1, 2, 6, 21, 75, 266, 938,\dots & A289597(?) \\ \hline
\R 1324 & 1, 1, 2, 6, 22, 84, 317, 1174, \dots & \\ \hline
\R 1342 & 1, 1, 2, 6, 22, 88, 367, 1568, \dots & A165538 \\ \hline
\R 1423, 3124 & 1, 1, 2, 6, 20, 68, 233, 805, \dots & A279557 \\ \hline
\R 1432 & 1, 1, 2, 6, 20, 71, 263, 1002, \dots & \\ \hline
\R 2143 & 1, 1, 2, 6, 19, 62, 207, 704, \dots & A026012 \\ \hline
\R 2314 & 1, 1, 2, 6, 23, 99, 450, 2109, \dots & \\ \hline
\R 2341 & 1, 1, 2, 6, 22, 91, 409, 1955, \dots & \\ \hline
\R 2413, 2431, 3241 & 1, 1, 2, 6, 22, 90, 395, 1823, \dots & A165546(?) \\ \hline
\R 3142 & 1, 1, 2, 5, 14, 42, 132, 429, \dots & A000108 \\ \hline
\R 3214 & 1, 1, 2, 6, 20, 72, 275, 1096, \dots & \\ \hline
\R 3412 & 1, 1, 2, 6, 22, 90, 396, 1840, \dots & \\ \hline
\R 3421 & 1, 1, 2, 6, 22, 92, 423, 2088, \dots & \\ \hline
\R 4123 & 1, 1, 2, 5, 14, 43, 143, 507, \dots & \\ \hline
\R 4132, 4213 & 1, 1, 2, 5, 15, 51, 188, 732, \dots & \\ \hline
\R 4231 & 1, 1, 2, 6, 22, 90, 396, 1841, \dots & \\ \hline
\R 4312 & 1, 1, 2, 5, 15, 51, 188, 733, \dots & \\ \hline
\R 4321 & 1, 1, 2, 5, 17, 66, 279, 1256, \dots & \\ \hline
\end{tabular}
\bigskip
\caption{} \label{tab:indLength4}
\end{table}

In this paper, we have discussed the enumeration of Fishburn permutations that avoid a pattern of size 3 or a pattern of size 4. In Section~\ref{sec:length3patterns}, we offer the complete picture for patterns of size 3, including the enumeration of indecomposable permutations.

Regarding patterns of size 4, we have proved the Wilf equivalence of eight permutation families counted by the Catalan numbers. We have also shown that $\F_n(1342)$ is enumerated by the binomial transform of the Catalan numbers. In general, there seems to be 13 Wilf equivalence classes of permutations that avoid a pattern of size 4, some of which appear to be in bijection with certain pattern avoiding ascent sequences (\oeis{A202061, A202062}). At this point in time, we don't know how the pattern avoidance of a Fishburn permutation is related to the pattern avoidance of an ascent sequence. It would be interesting to pursue this line of investigation.

Concerning indecomposable permutations, we leave the field open for future research. Note that Theorem~\ref{thm:Catalan3142} and Lemma~\ref{lem:invert_ind} imply
\[ |\F_n^{\textsf{ind}}(3142)| = C_{n-1}. \]
The study of other patterns is unexplored territory, and our preliminary data suggests the existence of 19 Wilf equivalence classes listed in Table~\ref{tab:indLength4}.

We are particularly curious about the class $\F_n^{\textsf{ind}}(2413)$ as it appears (based on limited data) to be equinumerous with the set $\Av_{n-1}(2413,3412)$, cf.~\oeis{A165546}.

%%%%%%%%%%%%%%%%%%%%%%%%%%%%%%%%%%%%%%%%%%%%

%%%%%%%%%%%%%%%%%%%%%%%%%%%%%%%%%%%%%%%%%%%%
\end{document}